\documentclass[10pt,a4paper,twoside]{article}
\usepackage{amssymb,amsthm}
\usepackage{amsmath}
\usepackage{setspace}
\usepackage[a4paper,top=33mm, bottom=27mm, left=20mm, right=20mm]{geometry}
\usepackage[small, margin=20pt]{caption}
\usepackage{url}
\usepackage{hyperref}

\usepackage{fancyhdr}
\allowdisplaybreaks[4]

\title{Random subgraphs make identification affordable}

\author{Florent Foucaud, Guillem Perarnau and Oriol Serra}
\date{\today}

\theoremstyle{plain}
\newtheorem{theorem}{Theorem}
\newtheorem{lemma}[theorem]{Lemma}
\newtheorem{proposition}[theorem]{Proposition}
\newtheorem{corollary}[theorem]{Corollary}
\newtheorem{observation}[theorem]{Observation}

\newtheorem{question}[theorem]{Question}

\theoremstyle{definition}

\newtheorem{definition}[theorem]{Definition}

\theoremstyle{definition}

\newtheorem*{claim}{Claim}

\newcommand{\rst}[1]{\ensuremath{{\mathbin\upharpoonright}%
\raise-.5ex\hbox{$#1$}}} 

\newcommand{\E}{\mathbb{E}}
\renewcommand{\P}{\mathcal{P}}

\newcommand{\eps}{\varepsilon}

\newcommand{\M}{\gamma^{\text{\tiny{ID}}}}

\newcommand{\Poly}{\mbox{Poly}}
\newcommand{\Sym}{\oplus}

\begin{document}

\pagenumbering{arabic}

\setcounter{section}{0}

\maketitle

\onehalfspace

\begin{abstract}
An identifying code of a graph is a dominating set which uniquely
determines all the vertices by their neighborhood within the code.
Whereas graphs with large minimum degree have small domination number,
this is not the case for the identifying code number (the size of a
smallest identifying code), which indeed is not even a monotone
parameter with respect to graph inclusion.

We show that every graph $G$ with $n$ vertices, maximum degree
$\Delta=\omega(1)$ and minimum degree
$\delta\geq c\log{\Delta}$, for some constant $c>0$, contains a large spanning
subgraph which admits an identifying code with size
$O\left(\frac{n\log{\Delta}}{\delta}\right)$. In particular, if
$\delta=\Theta(n)$, then $G$ has a dense spanning subgraph with
identifying code $O\left(\log n\right)$, namely, of asymptotically
optimal size. The subgraph we build is created using a probabilistic
approach, and we use an interplay of various random methods to analyze
it. Moreover we show that the result is essentially best possible,
both in terms of the number of deleted edges and the size of the
identifying code. 
\end{abstract}

\section{Introduction}
Consider any graph parameter that is not monotone with respect to
graph inclusion. Given a graph $G$, a natural problem in this context
is to study the minimum value of this parameter over all spanning
subgraphs of $G$. In particular, how many edge deletions are
sufficient in order to obtain from $G$ a graph with near-optimal value
of the parameter? Herein, we use random methods to study this question
with respect to the identifying code number of a graph, a well-studied
non-monotone parameter. An identifying code of graph $G$ is a set $C$ of
vertices which is a dominating set, and such that the closed
neighborhood within $C$ of each vertex $v$ uniquely determines
$v$.

Identifying codes were introduced in 1998 in~\cite{KCL98} and have
been studied extensively in the literature since then. We refer
to~\cite{biblio} for an on-line bibliography. One of the interests of
this notion lies in their applications to the location of threats in
facilities~\cite{UTS04} and error-detection in computer
networks~\cite{KCL98}. One can also mention applications to
routing~\cite{LTCS07}, to bio-informatics~\cite{HKSZ06} and to
measuring the first-order logical complexity of
graphs~\cite{KPSV04}. Let us also mention that identifying codes are
special cases of the more general notion of \emph{test covers} of
hypergraphs, see e.g.~\cite{DHHHLRS03,MS85} (test covers are also the
implicit object of Bondy's celebrated theorem on \emph{induced
  subsets}~\cite{B72}).

Let $G$ be a simple, undirected and finite graph. 
The \emph{open neighborhood} of a vertex $v$
in $G$ is the set of vertices in $V(G)$ that are adjacent to it, and
will be denoted $N_G(v)$.  The \emph{closed neighborhood} of a vertex
$v$ in $G$ is defined as $N_G[v]=N_G(v)\cup \{v\}$.
The degree of a vertex $u\in V(G)$, is defined as $d(v)= |N_G(v)|$.
Similarly, we define, for a set $S\subseteq V(G)$,
$N_G(S)=\bigcup_{v\in S}N_G(v)$ and $N_G[S]=\bigcup_{v\in S}N_G[v]$.
If two distinct vertices $u,v$
are such that $N[u]=N[v]$, they are called \emph{twins}. The symmetric
difference between two sets $A$ and $B$ is denoted by $A\Sym B$. 

Given a graph $G$ and a subset $C$ of vertices of $G$, $C$ is called a
\emph{dominating set} if each vertex of $V(G)\setminus C$ has at least
one neighbor in $C$. The set $C$ is called a \emph{separating set} of
$G$ if for each pair $u,v$ of vertices of $G$, $N[u]\cap C\neq
N[v]\cap C$ (equivalently, $(N[u]\Sym N[v])\cap C\neq\emptyset$). If
$x\in N[u]$, we say that $x$ \emph{dominates} $u$. If $x\in N[u]\Sym
N[v]$, we say that $x$ \emph{separates} $u,v$.

\begin{definition}
A subset of vertices of a graph $G$ which is both a dominating set and
a separating set is called an \emph{identifying code} of $G$.
\end{definition}

The following observation gives an equivalent condition for a set to
be an identifying code, and follows from the fact that for two
vertices $u,v$ at distance at least~3 from each other,
$N[u]\Sym N[v]=N[u]\cup N[v]$.

\begin{observation}\label{obs:dist2}
For a graph $G$ and a set ${C}\subseteq V(G)$, if ${C}$ is dominating
and $N[u]\cap {C}\neq N[v]\cap {C}$ for each pair of vertices $u,v$ at
distance at most two from each other, then $C$ is an identifying code
of the graph.
\end{observation}

The minimum size of a dominating set of graph $G$, its
\emph{domination number}, is denoted by $\gamma(G)$. Similarly, the
minimum size of an identifying code of $G$, $\M(G)$, is the
\emph{identifying code number} of $G$. It is known that for any
twin-free graph $G$ on $n$ vertices having at least one edge, we have:
$$\lceil\log_2(n+1)\rceil\le\M(G)\leq n-1.$$ The lower bound was proved
in~\cite{KCL98} and the upper bound, in~\cite{B01,GM07}. Both bounds are
tight and all graphs reaching these two bounds have been classified
(see~\cite{M06} for the lower bound and~\cite{FGKNPV10} for the upper
bound). Other papers studying bounds and extremal graphs for
identifying codes are e.g.~\cite{CHL07,FKKR10,FP12}.

In view of the above lower bound, we say that an identifying code $C$
of $G$ is \emph{asymptotically optimal} if
$$
|C| = O(\log{n})\;.
$$

The problem we will address in this paper is to deal with graphs that
have a large identifying code number, or are not even
identifiable. Our approach will consist in slightly modifying such a
graph in order to decrease its identifying code number and obtain an
asymptotically optimal identifying code, unless its domination number
prevents us from doing so.

One of the reasons for a graph to have a large identifying code number
is that it has a large domination number (this one being a monotone
parameter under edge deletion). For instance, we need roughly $n/3$
vertices to dominate all the vertices in a path of order $n$. When
this is the case, we cannot expect to decrease much the size of a
minimum identifying code by deleting edges from $G$, as the deletion
of edges cannot decrease the domination number.

However, there are many graphs with small domination number where the
identifying code number is very large~\cite{FGKNPV10,FP12}. Typically,
this phenomenon appears in graphs having a specific, ``rigid'',
structure. Supporting this intuition, Frieze, Martin, Moncel, Ruszink\'o and Smyth \cite{FMMRS07} have shown that the random graph $G(n,p)$ with $p\in (0,1)$,
admits an asymptotically optimal identifying code. In particular, they prove in~\cite{FMMRS07}  that
\begin{align*}
 \M(G(n,p))=(1+o(1)) \frac{2\log{n}}{\log{(1/q)}}\;,
\end{align*}
where $q=p^2+(1-p)^2$. This suggests that the lack of structure in
dense graphs implies the existence of a small identifying code.

\textbf{Our results and structure of the paper.} In Section~\ref{sec:main}, we prove our main result by selecting at
random a small set of edges that can be deleted to ``add some
randomness'' in the graph,

\begin{theorem}\label{thm:stronger}
For any graph $G$ on $n$ vertices ($n$ large enough) with maximum
degree $\Delta=\omega(1)$ and minimum degree $\delta\geq 66\log{\Delta}$, there
exists a subset of edges $F\subset E(G)$ of size
$$
|F|\leq 83n\log{\Delta}\;,
$$ 
such that 
$$
\M(G\setminus F)\leq 134\frac{n\log{\Delta}}{\delta}\;.
$$
\end{theorem}
Observe that when $\delta=\Theta(n)$, this result is similar to the one
in~\cite{FMMRS07}.

We then show in Section~\ref{sec:tight} that our result is
asymptotically best possible in terms of both the number of deleted
edges and of the size of the final identifying code for any graph with
$\Delta=\Poly(\delta)$. For smaller values of the minimum degree, we
prove that our result is almost optimal. We also show that the two
assumptions $\Delta=\omega(1)$ and $\delta\geq c\log{\Delta}$ for some
constant $c$ are necessary.

We present some consequences of our result in
Section~\ref{sec:consequences}. When considering the case of adding
edges to the graph, we get analogous (symmetric) results, showing that
every graph is a large spanning subgraph of some graph that admits a
small identifying code. This result also turns out to be tight. We
also describe an application to the closely related topic of
\emph{watching systems}.

The paper concludes with some final remarks and open problems.

\textbf{Our methods.} To show our results, we use the technique of
defining a suitable random spanning subgraph of $G$: we first randomly
choose a code $C$, and then we randomly delete edges among the edges
containing vertices of $C$. We then analyze the construction by
applying concentration inequalities and the use of the local lemma.

A similar approach has been used in the literature when considering
\emph{random subgraphs of a graph}: for any graph $G$, consider the
graph $G_p$ to be the subgraph of $G$ obtained by keeping \emph{each}
edge from $E(G)$ independently with probability $p$. The behavior of
random subgraphs of graphs $G$ with minimum degree $\delta$, inspired
by applications in epidemiology or social networks, has been widely
studied~\cite{ch2007,fk2012,kls2012,ks2013}.  A well-known instance of
this problem is the classical Erd\H os-R\'enyi random graph $G(n,p)$
where $G=K_n$, the complete graph of order $n$. In most of the cases,
it was shown that many similarities exist between $G_p$ and the random
graph $G(\delta,p)$.  The connectivity of a random subgraph of a graph,
where every edge has a different probability of being deleted, has
been studied in~\cite{A1993}. Our random subgraph model is adapted to
the analysis of identifying codes, and can be seen as a weighted
version of $G_p$.

\section{Main theorem}\label{sec:main}

In this section, we prove Theorem~\ref{thm:stronger}. We will need
some tools and lemmas.

\subsection{Important tools and lemmas}

In our proofs, we will repeatedly use the Chernoff inequality for the
sum of independent bounded random variables:

\begin{lemma}[{Chernoff inequality~\cite[Corollary~$A.1.14$]{AS00}}]\label{lem:chernoff2}
	Let $X_1,\dots,X_{N}$ be independent Bernoulli random variable with probability $p_i$ and
define $X=\sum_{i=1}^{N} X_i$.  Then , for all $\eps>0$,
$$\Pr(|X- \E(X)|\geq \eps \E(X))< 2e^{-c_{\eps}\E(X)}\;,$$
where
$$
c_{\eps}= \min \left\{(1+\eps)\log(1+\eps)-\eps, \frac{\eps^2}{2}\right\}\;.
$$
\end{lemma}

In what follows, for any set of vertices $B\subseteq V(G)$ and any
$v\in V(G)$, we let $N_G^B(v)= N_G(v)\cap B$ be the set of neighbors of $v$ in $B$. Analogously,  $N_G^B[v]= N_G[v]\cap B$. We denote by $d_B(v)=|N_G^B (v)|$,  the degree of $v$ within set $B$.

\begin{definition}\label{def:G(B,f)}
Given a graph $G$ and $B\subseteq V(G)$, a function $f: V(G)
\rightarrow\mathbb{R}^+\cup \{0\}$ is said to be $(G,B)$--{\it
  bounded} if for each vertex $u$, $f(u)\leq d_B(u)$ and for each pair
$u,v$ of vertices with $d_B(u)\geq d_B(v)$, $f(u)/d_B(u)\leq
f(v)/d_B(v)$. Given a $(G,B)$--bounded function $f$, we define the
random spanning subgraph $G(B,f)$ of $G$ as follows:
\begin{itemize}
\item $G(B,f)$ contains all edges of the subgraph $G[V(G)\setminus B]$ induced by $V(G)\setminus B$, and
\item each edge $uv$ incident with $B$ is independently chosen to be in $G(B,f)$ with probability $1-p_{uv}$, where
$$
p_{uv}=\frac{1}{4}\left(\frac{f(u)}{d_B(u)} + \frac{f(v)}{d_B(v)}\right)\;.
$$
Observe that, since $f(u)\leq d_B(u)$ for each vertex $u\in V(G)$, we have $p_{uv}\leq 1/2$.
\end{itemize}
\end{definition}

The next lemma gives an exponential upper bound on the probability
that two vertices of $G(B,f)$ are not separated by $B$. This lemma
is a crucial one in our main proof.

\begin{lemma}\label{lem:symmetric}
 Let $G$ be a graph, $B\subseteq V(G)$, and $f$ a $(G,B)$--bounded
 function. In the random subgraph $G(B,f)$, for every pair
 $u,v$ of distinct vertices with $d_B(u)\geq d_B(v)$, we have
\begin{align*}
\Pr\left(N_{G(B,f)}^B[u]=N_{G(B,f)}^B[v]\right)& \leq 
e^{-3 f(u)/16 }\;.
\end{align*}
\end{lemma}
\begin{proof}
Consider the following partition of $S=N^B_G[u]\cup N^B_G[v]$ into
three parts: $S_1$, the vertices of $B$ dominating $u$ but not $v$; $S_2$, the vertices of $B$
dominating $v$
but not $u$; and $S_3$, the vertices of $B$ dominating both $u$ and
$v$.

Let $D$ be the random variable which gives the size of the symmetric
difference of $N_{G(B,f)}^B[u]$ and $N_{G(B,f)}^B[v]$. The statement
of the lemma is equivalent to $\Pr(D=0) < e^{-3f(u)/16}$.

The random variable $D= |N_{G(B,f)}^B[u]\oplus N_{G(B,f)}^B[v]|$ can be written as the sum of
independent Bernoulli variables
$$
D=\sum_{w\in S} D_w\;,
$$
where $D_w=1$ if and only if $w$ dominates precisely one of the two vertices $u$ or
$v$ in $G(B,f)$. Therefore, fro any $w\notin\{u,v\}$,
$$
\Pr (D_w=1)=\left\{ \begin{array}{ll}  1-p_{uw} & w\in S_1\\  1-p_{vw} & w\in
S_2\\p_{uw}(1-p_{vw})+p_{vw}(1-p_{uw})&w\in S_{3}\end{array}\right.
$$

Since we want to bound from above the probability that $D=0$, we can always assume that $u,v\notin
N_{G(B,f)}^B[u]\Sym  N_{G(B,f)}^B[v]$. 
Recall that $d_B(u)\geq d_B(v)$. By the definition of a $(G,B)$--bounded function, we have that
$p_{uw}\leq p_{vw}$ for each $w\in S_3$. Since $x(1-x)$ has a unique maximum at $x=1/2$ and $p_{uw},
p_{vw}\le 1/2$, we also have:
\begin{align}\label{eq:3}
p_{vw}(1-p_{uw})\geq p_{uw}(1-p_{uw}) \geq \frac{f(u)}{4d_B(u)}\left( 1-
\frac{f(u)}{4d_B(u)}\right)=g(u)\;,
\end{align}
for each $w\in S_3$.

For $w\in S$, denote by $q_w$  the parameter of the Bernoulli random variable $D_w$.
Then,
\begin{align}\label{eq:1}
\E(D) &\geq\sum_{w\in N^B_G(u)} q_w  \nonumber \\
 & =  \sum_{w\in S_1} q_w+ \sum_{w\in S_3} q_w \nonumber \\
 & = \sum_{w\in S_1} (1-p_{uw})+ \sum_{w\in S_3}  \left(p_{uw}(1-p_{vw})+p_{vw}(1-p_{uw})\right)
\nonumber \\
 &\geq \sum_{w\in S_1} p_{uw}(1-p_{uw})+ \sum_{w\in S_3}   p_{uw}(1-p_{uw}) \nonumber \\
& \geq g(u) d_B(u) \nonumber\\ 
&= \frac{f(u)}{4}\left(1-\frac{f(u)}{4d_B(u)}\right) \nonumber\\
&\geq  \frac{3}{16} f(u)\;.
\end{align}

Finally, we have that
\begin{align*}
\Pr(D=0)&= \prod_{w\in S} (1-q_{w})\leq  e^{-\sum_{w\in S} q_w}
 =  e^{-\E(D)}
 \leq e^{-3 f(u)/16 }\;,
\end{align*}
and the lemma follows.
\end{proof}

In the proof of our main result, we will use the following version of
the Lov\'asz local lemma, which can be found in e.g. \cite[Corollary
  $5.1.2$]{AS00} (the lower bound on
$\Pr(\bigcap_{i=1}^M\overline{E_i})$ can be derived from the general
local lemma, see \cite[Lemma
  $5.1.1$]{AS00}, by setting $x_i=e\cdot p_{LL}$).

\begin{lemma}[Symmetric Local Lemma]\label{lem:LL}
Let $\mathcal{E}=\left\{E_1,\ldots,E_M\right\}$ be a set of (typically
``bad'') events such that each $E_i$ is mutually independent of
$\mathcal{E}\setminus(\mathcal{D}_i\cup\left\{E_i\right\})$ for some
$\mathcal{D}_i\subseteq \mathcal E$. Let $d_{LL}=|\mathcal D_i|$, and
suppose that there exists a real $0<p_{LL}<1$ such that, for each $1\le i\le
M$,
\begin{itemize}
\item $\Pr(E_i)\leq p_{LL}$, and
\item $e\cdot p_{LL}\cdot (d_{LL}+1)\leq 1$.
\end{itemize}

Then $\Pr(\bigcap_{i=1}^M\overline{E_i})\ge(1-e\cdot p_{LL})^M>0$.
\end{lemma} 

\subsection{Proof of the main result}

We are now ready to prove the main theorem.

\begin{proof}[Proof of Theorem~\ref{thm:stronger}]
The proof is structured in the following steps:
\begin{enumerate}
\item We select a set $C$ at random, where each vertex is selected
  independently with probability $p$. Using the Chernoff inequality,
  we estimate the probability of the event $A_C$ that $C$ is small
  enough for our purposes. From $C$, we construct the spanning
  subgraph $G(C,f)$ of $G$ as given in Definition~\ref{def:G(B,f)},
  for some suitable function $f$.
\item We use the local lemma (Lemma~\ref{lem:LL}) and
  Lemma~\ref{lem:symmetric} to bound from below the probability that
  the following events (whose conjunction we call $A_{LL}$) hold
  jointly: (i) in $G(C,f)$, each pair of vertices that are at distance
  at most~2 from each other are separated by $C$; and (ii) for each
  such pair and each member of this pair in $G$, its degree within
  $C$ in $G$ is close to its expected value $d(v)p$. We show that with
  nonzero probability, $A_C$ and $A_{LL}$ hold jointly.
\item We find a dominating set $D$ of $G$ with $|D|=O(|C|)$; by
  Observation~\ref{obs:dist2}, if $A_{LL}$ holds, then $C\cup D$ is an
  identifying code.
\item Finally, we show that, subject to $A_C$ and $A_{LL}$, the
  expected number of deleted edges is as small as desired.
\end{enumerate}

\textbf{Step 1. Constructing $C$ and $G(C,f)$}

Let $C\subseteq V(G)$ be a subset of vertices, where each vertex $v$ in
$G$ is chosen to be in $C$ independently with probability
	$$
	p=\frac{66\log{\Delta}}{\delta}\;.
	$$
Observe that $p\leq 1$ since $\delta\geq 66\log{\Delta}$.

Consider the random variable $|C|$ and recall that $\E(|C|)=n
p$. 

Define $A_C$ to be the event that
 \begin{align}\label{eqn:einstein}
     |C|\leq 2np =\frac{132n\log{\Delta}}{\delta}.\tag{$A_C$}
  \end{align}

Since the choices of the elements in $C$ are done independently, by
setting $\eps=1$ in Lemma~\ref{lem:chernoff2}, notice that $c_\eps >1/3$, we have
 \begin{align}\label{eqn:A_C}
\Pr(\overline{A_C}) < e^{-\frac{22n\log{\Delta}}{\delta}}\;.
 \end{align}

We let 
$$
f(u)=\min\left(66\log\Delta,d_C(u)\right).
$$ 

Observe that $f$ is $(G,C)$--bounded. We construct $G(C,f)$ as the
random spanning subgraph of $G$ given in
Definition~\ref{def:G(B,f)}, where each edge $uv$ incident to a vertex
of $C$ is deleted with probability $p_{uv}$.

\textbf{Step 2. Applying the local lemma}

Let $u,v$ be a pair of vertices at distance at most~2 in $G$. We
define the following events:
\begin{itemize}
\item $A_{uv}$ is the event that there exists a vertex $w\in\{u,v\}$ such that the degree of $w$ within $C$ is deviating from its expected value $d(w)
p$ by half, i.e. $|d_C(w)-d(w) p| \geq \frac{d(w) p}{2}$;
\item $B_{uv}$ is the event that $N_{G(C,f)}^C[u]=N_{G(C,f)}^C[v]$;
\item $E_{uv}$ is the event that $A_{uv}$ or $B_{uv}$ occurs;
\item $A_{LL}$ is the event that no event $E_{uv}$ occurs.
\end{itemize}

In order to apply the Local Lemma, we wish to upper bound the
probability of $E_{uv}$.  We have:
\begin{align*}
\Pr(E_{uv})&\leq \Pr(A_{uv})+\Pr(B_{uv})\\
&=\Pr(A_{uv})
+\Pr(B_{uv}|A_{uv})\cdot\Pr(A_{uv})+\Pr(B_{uv}|\overline{A_{uv}})\cdot\Pr(\overline{A_{ uv}})\;.
\end{align*}

Let us upper bound $\Pr(A_{uv})$. We use
Lemma~\ref{lem:chernoff2} with $\eps=1/2$. Observe that
$c_{\eps}>\frac{1}{10}$, and thus
\begin{align*}
\Pr(A_{uv})& <\Pr\left(|d_C(u)-d(u) p|\geq \frac{d(u) p}{2}\right)+\Pr\left(|d_C(v)-d(v) p|\geq \frac{d(v) p}{2}\right)\\
&\leq  2e^{-\frac{1}{10} d(u)p} +2e^{-\frac{1}{10} d(v)p}\\
&=  2e^{-\frac{66d(u)\log\Delta}{10\delta}} +2e^{-\frac{66d(v)\log\Delta}{10\delta}}\\
& \leq 4e^{-\frac{33\log\Delta}{5}}\\
& \leq 4\Delta^{-\frac{33}{5}}\;.
\end{align*}
Next, we give an upper bound for $\Pr(B_{uv}|\overline{A_{uv}})$. For such a purpose, we apply
Lemma~\ref{lem:symmetric} with $B=C$ and
$f(u)=\min(66\log\Delta,d_C(u))$. Observe that $f$ is $(G,C)$--bounded.
Since $A_{uv}$ does
not hold, we know that $d_C(u)$ and $d_C(v)$ are large enough, i.e. for
$w\in\{u,v\}$, $d_C(w)\geq\frac{d(w) p}{2}\geq\frac{\delta
  p}{2}=33\log\Delta$; thus $f(u),f(v)\geq 33\log\Delta$. We have:
\begin{align}
\Pr(B_{uv}|\overline{A_{uv}})\leq e^{-\frac{3\cdot 33\log\Delta}{16}}\leq\Delta^{-\frac{99}{16}}\;.
\end{align}
The probability that the event $E_{uv}$ holds is
\begin{align*}
\Pr(E_{uv})\leq &
\Pr(A_{uv})
+\Pr(B_{uv}|A_{uv})\cdot\Pr(A_{uv})+\Pr(B_{uv}|\overline{A_{uv}})\cdot\Pr(\overline{A_{uv } })\\
& \leq  4\Delta^{-\frac{33}{5}}+1\cdot 4\Delta^{-\frac{33}{5}} + \Delta^{-\frac{99}{16}} \cdot 1\\
& \leq 2\Delta^{-\frac{99}{16}}=p_{LL}\;,
\end{align*}
where we used $\Delta=\omega(1)$.

We now note that each event $E_{uv}$ is mutually independent of all but at most $2\Delta^6$ events
$E_{u'v'}$. Indeed, $E_{uv}$ depends on the random variables determining the existence of the edges
incident to $u$ and $v$. This is given by probabilities $p_{uw}$ and $p_{vw}$ that depend on
$d_C(w)$, where $w$ is at distance at most one from either $u$ or $v$. Thus, $E_{uv}$ depends only
on the vertices at distance at most two from either $u$ or $v$ belonging to $C$. In other words,
$E_{uv}$ and $E_{u'v'}$ are mutually independent unless there exist a vertex $w$ at distance at most
two from both pairs; in other words, $d(\{u,v\},\{u',v'\})\leq 4$. Hence, there are at most
$2\Delta^4$ choices for the vertex among $\{u',v'\}$ that is closest from $\{u,v\}$ (say $u'$), and
at most $\Delta^2$ additional choices for $v'$, since $d(u',v')\leq 2$.

Therefore, we can apply Lemma~\ref{lem:LL} if
$$
e\cdot 2\Delta^{-\frac{99}{16}}\cdot (2\Delta^6+1) \leq 1\;,
$$
which holds since $\Delta=\omega(1)$.

Now, by Lemma~\ref{lem:LL} and since there are at most $\frac{n\Delta^2}{2}$
events $E_{uv}$ (one for each pair of vertices at distance at most~2 from each other) and
$p_{LL}=2\Delta^{-\frac{99}{16}}$,
\begin{align}\label{eqn:A_LL}
\Pr(A_{LL})&\geq (1-e\cdot p_{LL})^M \geq e^{-2e\cdot p_{LL}M}\geq
e^{-2en\Delta^{2-\frac{99}{16}}}\;,
\end{align}
where we have used $(1-x)= e^{-x (1-O(x))}\geq e^{-2x}$, if $x=o(1)$.

\textbf{Step 3. Revealing the identifying code}

Let us lower bound the probability that both $A_C$ and $A_{LL}$ hold,
by using Inequalities~\ref{eqn:A_C} and~\ref{eqn:A_LL}:
\begin{align*}
\Pr(A_C\cap A_{LL})&\geq\Pr(A_{LL})-\Pr(\overline{A_C})\\
& \geq e^{-2en\Delta^{2-\frac{99}{16}}}-e^{-\frac{22n\log{\Delta}}{\delta}}\;,
\end{align*}
which is strictly positive if
$$
\frac{22\log\Delta}{\delta} > 2e\Delta^{2-\frac{99}{16}}\;,
$$ 
which holds since $n$ is large (and hence $\Delta=\omega(1)$ is
large too), and $\delta\leq \Delta$.

Hence, there exists a set $C$ of size $132\frac{n\log{\Delta}}{\delta}$
such that all vertices at distance~2 from each other are separated by
$C$, and such that the degree in $C$ of all vertices is large.

In order to build an identifying code, we must also make sure that all
vertices are dominated. It is well-known that for any graph $G$,
$\gamma(G) \leq (1+o(1))\frac{n\log{\delta}}{\delta}$ (see
e.g. \cite[Theorem~$1.2.2$]{AS00}). Hence, we select a dominating set
$D$ of $G$ with size $(1+o(1))\frac{n\log{\delta}}{\delta}$. Then, by
Observation~\ref{obs:dist2}, $C\cup D$ is an identifying code of size
at most 
$$(132+1+o(1))\frac{n\log{\Delta}}{\delta}\leq
134\frac{n\log{\Delta}}{\delta}.$$

\textbf{Step 4. Estimating the number of deleted edges}

Let $Y=|E(G)\setminus E(G(C,f))|$ be the number
of edges we have deleted from $G$ to obtain $G(C,f)$.
Recall that each edge $uv\in E(G)$ is deleted independently from $G$ with probability 
$$
p_{uv}=\frac{1}{4}\left(\frac{f(u)}{d_C(u)}+\frac{f(v)}{d_C(v)}\right)\;,
$$
if one of its endpoints is in $C$.

Since  $\Pr(A_C\cap A_{LL})> 0$, there is a small identifying code of $G$ obtained by deleting at most  $\E(Y|A_C\cap A_{LL})$ edges. We next give an upper bound for $\E(Y|A_C\cap A_{LL})$. If both  $A_C$ and $A_{LL}$ hold, then 
$$
p_{uv}\leq \frac{1}{4}\left(\frac{66\log{\Delta}}{d_C(u)}+\frac{66\log{\Delta}}{d_C(v)}\right)\;.
$$
The expected number of deleted edges is
$$
\E(Y|A_C\cap A_{LL}) =\sum_{\substack{uv\in E(G)\\\left(\{u,v\}\cap C\right)\neq\emptyset}} p_{uv}\;.
$$
Observe that in order to estimate this quantity, we can split the two
additive terms in each $p_{uv}$: for every $u\notin C$, we sum all the
terms $\frac{66\log\Delta}{4d_C(u)}$ for all $v\in C$ being
neighbors of $u$; for every $u\in C$, we sum all the terms
$\frac{66\log\Delta}{4d_C(u)}$ for all $v\in V(G)$ being
neighbors of $u$.
\begin{align*}
	\E(Y|A_C\cap A_{LL}) &\leq \frac{1}{4}\left(\sum_{u\notin C} \sum_{v\in N^C_G(u)}
\frac{66\log\Delta}{d_C(u)}+ \sum_{u\in C}
\sum_{v\in N_G(u)} \frac{66\log\Delta}{d_C(u)}\right) \\
&\leq\frac{1}{4}\left(\sum_{u\notin C} d_C(u)\frac{66\log\Delta}{d_C(u)}+\sum_{u\in C}d(u)
\frac{66\log\Delta}{d_C(u)}\right)\\
&\leq\frac{1}{4}\left( |V(G)\setminus C|\cdot 66\log\Delta + \sum_{u\in C}
2\frac{66\log\Delta}{p}\right)\\
&\leq\frac{1}{4}\left(n\cdot 66\log\Delta + 2|C|\delta\right)\\
 &\leq \frac{66 n\log\Delta+ 264 n\log\Delta}{4}\\
& \leq 83n\log{\Delta}\;,
\end{align*}
where we used the fact (implied by $A_{LL}$) that for any vertex $v$,
$\frac{d(v)p}{2}\leq d_C(v)$ at the second line, and that $A_C$
implies $|C|\leq 132\frac{n\log{\Delta}}{\delta}$ at the fifth
line.

Summarizing, we have  shown the existence of a small identifying code in a 
spanning subgraph of $G$ obtained by deleting at most $\E(Y|A_C\cap
A_{LL})$ edges from $G$, which completes the proof.
\end{proof}

\section{Asymptotic optimality of Theorem~\ref{thm:stronger}}\label{sec:tight}

In this section, we discuss the optimality of
Theorem~\ref{thm:stronger}, first with respect to the size of the
constructed code and the number of deleted edges, and then with
respect to the hypothesis $\Delta=\omega(1)$ and  $\delta\geq 66\log\Delta$.

\subsection{On the size of the code and the number of deleted edges}

Charon, Honkala, Hudry and Lobstein showed that deleting an edge from
$G$ can decrease by at most $2$ the identifying code number of a
graph~\cite{CHHL13}. That is, for any graph $G$ and any edge $uv \in
E(G)$,
$$
\M(G)\leq  \M(G\setminus uv)+2\;.
$$

This directly implies that for every graph with linear identifying
code number, one needs to delete a subset $F$ of at least $\Omega(n)$
edges, to get a graph with $\M(G\setminus F)=o(n)$.

We will show that, indeed, one needs to delete at least
$\Omega(n\log{n})$ edges from the complete graph to get a graph with
an asymptotically optimal identifying code. Using this, we will derive
a family of graphs with arbitrary minimum degree $\delta$, that
asymptotically attains the bound of Theorem~\ref{thm:stronger}, both
in number of edges and size of the minimum code, when
$\Delta=\Poly(\delta)$.

First of all, we prove that every graph with an asymptotically optimal
identifying code cannot contain too few edges.

\begin{lemma}\label{lem:sparse}
 For any $M'\geq 0$, there exists a constant $c_0>0$ such that any
 graph $G$ with $\M(G)\leq M'\log{n}$ contains at least $c_0 n\log{n}$
 edges.
\end{lemma}
\begin{proof}
 Set $\alpha_0$ as the smallest positive root of
\begin{align}\label{eq:cond_alpha}
 f(\alpha) =\alpha \log{\left(\frac{M'+\alpha}{\alpha} e\right)} - 1/2\;.
\end{align}
Note that $f(\alpha)$ is well-defined since $\lim_{\alpha\to 0}
f(\alpha) = -1/2$ and $f(1)=\log(M'+1)+1/2>0$.

Suppose by contradiction that there exists a graph $G$ containing less
than $c_0 n\log{n}$ edges, with $c_0=\alpha_0 /4$, that admits an
identifying code $C$ of size at most $M'\log{n}$.  Let $U$ be the
subset of vertices of degree at least $\alpha_0 \log{n}$. Notice that
$$
|U|\leq \frac{2|E(G)|}{\alpha_0 \log{n}} \leq \frac{2c_0}{\alpha_0}n = \frac{n}{2}\;.
$$ 


Since $|C|\leq M'\log{n}$ and any $v\in V(G)\setminus U$ has degree smaller than $\alpha_0 \log{n}$,
the number of possible nonempty sets
$N_G[v]\cap C$, is smaller than
\begin{align*}
\sum_{i=1}^{\alpha_0 \log{n}} \binom{|C|}{i} &\leq \binom{M'\log{n}+\alpha_0\log{n}}{\alpha_0
\log{n}} \\
&\leq \left(\frac{(M'+\alpha_0)e}{\alpha_0}\right)^{\alpha_0 \log{n}} \\ 
&= n^{\alpha_0  \log{\left(\frac{M'+\alpha_0}{\alpha_0} e\right)}} \\
&= \sqrt{n}\;.
\end{align*}
where we have used that $\binom{a}{b}\leq \left(\frac{ae}{b}\right)^b$ for
the second inequality and the fact that $\alpha_0$ is a root
of~\eqref{eq:cond_alpha} for the last one.

Since $|V(G)\setminus U|\geq n/2$ there must be at least two vertices
$v_1,v_2\in V(G)\setminus U$ such that $N_G[v_1]\cap C=N_G[v_2]\cap C$,
and thus $C$ cannot be an identifying code, a contradiction.
\end{proof}

The following lemma relates the identifying code number of a graph $G$ to
the one of its complement $\overline{G}$.
\begin{lemma}\label{lem:complement}
 Let $G$ be a twin-free graph. If
 $\overline{G}$ is twin-free, then
$$
\M(\overline{G})\leq 2 \M(G)\;.
$$
\end{lemma}
\begin{proof}
 Let $C_0$ be a minimum identifying code of $G$. We will show that
 there exists a set $C_1$ of size at most $\M(G)-1$ and a special
 vertex $v$, such that $C=C_0\cup C_1\cup \{v\}$ is an identifying
 code of $\overline{G}$.

For the sake of simplicity, we define the following relation. Two
vertices $u, v\in V(G)$ are in relation with each other if and only if
$N_{G}(u)\cap C_0= N_{G}(v)\cap C_0$ and $u\not\sim v$
(i.e. considering $C_0$ in $G$, $u,v$ are separated by one of $u,v$).
This will be denoted as $u\equiv_G v$. It can be checked that this
relation is an equivalence relation.
\begin{claim}
 Every pair of distinct vertices $u\not\equiv_G v$ is separated by
 $C_0$ in $\overline{G}$.
\end{claim}
\begin{proof}
 By the definition of $\equiv_G$, either $N_{G}(u) \cap
 C_0 \neq N_{G}(v)\cap C_0$ or $u\sim v$.

If $N_{G}(u) \cap C_0 \neq N_{G}(v)\cap C_0$, there exists $w\in C_0$
(and $w\notin\{u,v\}$) such that $w\in N_{G}(u)\Sym N_{G}(v)$.  Then,
$w\in N_{\overline{G}}(u)\Sym N_{\overline{G}}(v)$, hence $w$ still
separates $u,v$ in $\overline{G}$.

If $N_{G}(u) \cap C_0 = N_{G}(v)\cap C_0$, then $u\sim v$. If at least
one of them belongs to $C_0$, then this vertex separates $u,v$ in
$\overline{G}$.  Otherwise, $u, v\notin C_0$ and we have $N_{G}(u)
\cap C_0=N_{G}[u] \cap C_0$ and $N_{G}[v]\cap C_0 = N_{G}(v)\cap
C_0$. Hence $N_{G}[u] \cap C_0=N_{G}[v] \cap C_0$. But then $C_0$ does
not separate $u,v$ in $G$, a contradiction.
\end{proof}

In particular, this implies that any vertex in an equivalence class of
size one is separated by $C_0$ from all other vertices in $\overline{G}$.

\begin{claim}
 If $u\equiv_G v$ and both $u,v\notin C_0$, then $u=v$.
\end{claim}
\begin{proof}
 Since $u,v\notin C_0$, $N_{G}[u] \cap C_0=N_{G}(u) \cap C_0$ and
 $N_{G}[v] \cap C_0=N_{G}(v) \cap C_0$. Using that they are
 equivalent, we have that $N_{G}[u] \cap C_0 =N_{G}[v] \cap
 C_0$. Since $C_0$ is an identifying code of $G$, we must have $u=v$.
\end{proof}

\begin{claim}
 Let $U=\{u_1,\dots, u_s\}$ be an equivalence class of $\equiv_G$. Then all the
 pairs in $U$ can be separated in $\overline{G}$ by using $s-1$
 vertices.
\end{claim}
\begin{proof}
We will prove the claim by induction. For $s=2$ it is clearly true: since
$\overline{G}$ is twin-free, we can select $w\in N_{\overline{G}}[u_1]\Sym
N_{\overline{G}}[u_2]$, and $w$ separates $u$ and $v$ in $\overline{G}$.

For any $s>2$, consider the vertices $u_1,u_2 \in U$ and let $w\in
N_{\overline{G}}[u_1]\Sym N_{\overline{G}}[u_2]$. Since $U$ forms a
clique in $\overline{G}$, $w\notin U$. Then $w$ splits the set $U$ into $U_1$, the set
of vertices of $U$ adjacent to $w$ in $\overline{G}$, and $U_2$, the
set of vertices in $U$ non-adjacent to $w$ in $\overline{G}$. Let
$|U_1|=s_1$ and $|U_2|=s_2$; by construction, $s_1,s_2<s$.

Now, the pairs of vertices of $U$ with one vertex
from $U_1$ and one vertex from $U_2$ are separated by $w$. By
induction, the pairs of vertices in $U_1$ can be separated using
$s_1-1$ vertices and the ones in $U_2$ using $s_2-1$. Thus we need at
most $(s_1-1)+(s_2-1) +1 = s-1$ vertices to separate all the pairs of
vertices in $U$.
\end{proof}

From the previous claims, it is straightforward to deduce that there
is a set $C_1$ of size at most $|C_0|-1$ vertices that separates all
the pairs in $\overline{G}$ that are not separated by $C_0$.

Eventually, there might be a unique vertex $v$ such that
$N_{\overline{G}}[v]\cap (C_0\cup C_1)= \emptyset$ (if there were two
such vertices, they would not be separated by $C_0\cup C_1$, a
contradiction). Hence, $C=C_0\cup C_1 \cup \{v\}$ is an identifying
code of $\overline{G}$ of size at most $2|C_0|=2\M(G)$.
\end{proof}

\begin{proposition}\label{prop:optimal}
For any $M\geq 0$, there exists a constant $c>0$ such that for any set
of edges $F\subset E(K_n)$ satisfying $\M(K_n\setminus F)\leq
M\log{n}$, $|F|\geq c n\log{n}$.
\end{proposition}
\begin{proof}
Set $M'=M/2$ and let $c=c_0$ be the constant given by
Lemma~\ref{lem:sparse} for this $M'$.  Suppose that there exists a set $F$
of edges, $|F|< c n\log{n}$ such that $G=K_n\setminus F$ satisfies
$\M(G)\leq M\log{n}$. By Lemma~\ref{lem:complement}, the graph
$\overline{G}$ admits an identifying code of size at most $2M\log{n}= M'
\log{n}$. By Lemma~\ref{lem:sparse}, we get a contradiction.
\end{proof}

Using the former proposition, for any $\delta$ we can provide an
example of a graph with minimum degree $\delta$ for which the result
of Theorem~\ref{thm:stronger} is asymptotically tight when assuming
that $\Delta=\Poly(\delta)$.

For any $\delta>0$, consider the graph $H_\delta$ to be the
disjoint union of cliques of order $\delta+1$.  We may assume that $\delta+1$
divides $n$ for the sake of simplicity. Denote by $H_\delta^{(1)},\dots,
H_\delta^{(s)}$, $s=\frac{n}{\delta+1}$, the cliques composing $H_\delta$.

Since $H_\delta^{(i)}$ is a connected component, an asymptotically
optimal identifying code for $H_\delta$ must also be asymptotically
optimal for each $H_\delta^{(i)}$. By Proposition~\ref{prop:optimal},
we must delete at least $\Omega(\delta\log{\delta})$ edges from
$H_\delta^{(i)}$ to get an identifying code of size $O(\log{\delta})$.

Thus, one must delete at least $\Omega(s \delta\log{\delta}) =
\Omega(n\log{\delta})$ edges from $H_\delta$ to get an optimal
identifying code.

\begin{corollary}\label{cor:disjointcliques}
For any $\delta= \omega(1)$ and any $M\geq 0$, there exists a constant
$c>0$ such that for any set of edges $F\subset E(H_\delta)$ satisfying
$\M(H_\delta\setminus F)\leq M\frac{n\log{\delta}}{\delta}$, we have $|F|\geq
c n\log{\delta}$.
\end{corollary}

We remark that a connected counterexample can also be constructed from
$H_\delta$ by connecting its cliques using few edges, without
affecting the above result.


Corollary~\ref{cor:disjointcliques} implies that
Theorem~\ref{thm:stronger} is asymptotically tight when
$\Delta=\Poly(\delta)$, since in that case
$\log{\Delta}=O(\log{\delta})$.  However, when $\delta$ is
sub-polynomial with respect to $\Delta$, we do not know if
Theorem~\ref{thm:stronger} is asymptotically tight.

\subsection{On the hypothesis}

We conclude this section by discussing the necessity of  the hypothesis
$\Delta=\omega(1)$ and $\delta\geq 66 \log{\Delta}$ in
Theorem~\ref{thm:stronger}.

First note that, if $\Delta$ is bounded by a constant, we need at
least $\tfrac{n}{\Delta+1}=\Theta(n)$ vertices to dominate $G$. Thus,
no code of size smaller than
$\Theta(n)$ can be obtained by deleting edges of the graph.

On the other hand, the condition  $\delta\geq 66
\log{\Delta}$ in Theorem~\ref{thm:stronger}, is also necessary (up to
a constant factor) as can be deduced from the following proposition.
\begin{proposition}
For arbitrarily large values of $\Delta$, there exists a graph $G$
with maximum degree $\Delta$ and minimum degree
$\delta=\frac{\log_2{\Delta}}{2}$ such that, for any spanning subgraph
$H\subseteq G$,
$$
\M(H) = (1-o(1))n\;.
$$
\end{proposition}
\begin{proof}
Consider the bipartite complete graph $G=K_{r,s}$ where $s=2^{2r}$. Denote by $V_1$ the stable set of size $r$ and by $V_2$ the stable set of size $s$.
Observe that  $\delta= r= \frac{\log_2{s}}{2} =  \frac{\log_2{\Delta}}{2}$. 

For any given twin-free spanning subgraph $H\subseteq G$, let $C\subseteq V(G)$ be an identifying
code of $H$. Let us show that most of the vertices in $V_2$ must be in $C$. Let $S\subseteq V_2$
be the subset of vertices in $V_2$ that are not in the code. Thus, for any $u\in S$, $N_C[u]=
N_C(u)$. Observe that $N_C(u)\subseteq V_1$, and hence, there are at most $2^{r}$ possible
candidates for such $N_C(u)$. Since $C$ is dominating and separating all the pairs in $S$, all the
subsets $N_C(u)$ must be non empty and different, which implies, $|S|<2^{r}$. 
Hence, we have 
$$
|C|\geq |V_2\setminus S|\geq  2^{2r}-2^r = (1-o(1))2^{2r} =(1-o(1))n\;.
$$
\end{proof}

\section{Consequences of our results}\label{sec:consequences}

We now describe consequences of our results on the case when we want
to \emph{add} edges to a graph to decrease its identifying code
number, and to the notion of watching systems.

\subsection{Adding edges}\label{sec:adding}

In the previous sections, we have studied how much can
the identifying code number decrease when we delete few edges from
the original graph. In this section, we discuss the symmetric question of how much can
the addition of edges help to decrease this parameter.

The question of how much can a parameter decrease when deleting/adding
edges has been already studied for some monotone parameters. However,
if the parameter is monotone, only one of either deleting or adding,
can help to decrease it. One of the interesting facts of studying the
identifying code number is that, since it is a non-monotone parameter,
we can have similar results for both procedures.

As before, let $G$ be a graph with maximum degree $\Delta$ and minimum degree $\delta$. We aim to
find a set of edges $F$ with $F \cap E(G)=\emptyset$ such that $\M (G\cup F)$ is small. This set $F$
will be provided by applying Theorem~\ref{thm:stronger} to the graph $\overline{G}$, that has
maximum degree $\Delta(\overline{G})=n-1- \delta$ and minimum degree $\delta(\overline{G})
=n-1-\Delta$. Thus, it will have size
$$
|F|= O\left(n\log{\Delta(\overline{G})}\right),
$$
and 
$$
\M( \overline{G}\setminus F) =
O\left(\frac{n\log{\Delta(\overline{G})}}{\delta(\overline{G})}\right)\;.
$$

Since $\overline{G}\setminus F= \overline{G\cup F}$, we have the
following corollary of Theorem~\ref{thm:stronger} and
Lemma~\ref{lem:complement}.

\begin{corollary}\label{cor:adding}
 For any graph $G$ on $n$ vertices with minimum degree $\delta= n-\omega(1)$ and maximum degree
$\Delta$ such that $n-\Delta\geq 66\log{(n-\delta)}$, there exists a set of edges $F$ with $F
 \cap E(G)=\emptyset$ of size
$$
|F|= O\left(n\log{(n-\delta)}\right)\;,
$$ 
such that
$$
\M(G\cup F) =  O\left(\frac{n\log{n}}{n-\Delta}\right)\;.
$$
\end{corollary}

This result is also asymptotically tight. Otherwise, by using again
Lemma~\ref{lem:complement}, we could translate our case to the case of deleting
edges and we would get a contradiction with the optimality of
Theorem~\ref{thm:stronger}.

\subsection{Watching systems}\label{sec:watching}

The result of Theorem~\ref{thm:stronger} has a direct application for
\emph{watching systems}, which are a generalization of identifying
codes~\cite{ACHL10,ACHL12}. In a watching system, we can place on each
vertex $v$ a set of \emph{watchers}. To each watcher $w$ placed on
$v$, we assign a nonempty subset $Z(w)\subseteq N[v]$, its
\emph{watching zone}. We now ask each vertex to belong to a unique and
nonempty set of watching zones; the minimum number of watchers that
need to be placed on the vertices of $G$ to obtain a watching system
is the \emph{watching number} $w(G)$ of $G$.

It is clear from the definition that $\gamma(G)\leq w(G)\leq\M(G)$,
since the vertices of any identifying code form a watching system
(where the watching zones are the closed neighborhoods). In fact,
even the following holds:

\begin{observation}\label{obs:watch}
For any twin-free graph $G$, $w(G)\leq\min\{\M(H), \mbox{ where $H$ is a spanning
  subgraph of $G$}\}$. Indeed, consider the spanning subgraph $H_0$ of
$G$ with smallest identifying code number, and define the watching
system to be the vertices of an optimal identifying code of $H_0$,
with the watching zones being the closed neighborhoods in $H_0$.
\end{observation}

In~\cite[Theorems~$2$ and~$3$]{ACHL10}, the authors propose the
following upper bound for graphs with given maximum degree:

\begin{theorem}[\cite{ACHL10}]\label{thm:watchDelta}
 Let $G$ be a graph with maximum degree $\Delta$, then
$$\lceil \log_2(n+1) \rceil \leq w(G) \leq \gamma(G) \lceil\log_2(\Delta+2)\rceil\;.$$
\end{theorem}

Note that for any values of parameters $\gamma$ and $\Delta$, the
upper bound from the above theorem is tight for the graph consisting
of $\gamma$ disjoint copies of a star on $\Delta+1$ vertices.

It is well-known (see e.g. \cite[Theorem~$1.2.2$]{AS00}) that the
domination number of a graph with minimum degree $\delta$ satisfies
$$
\gamma(G) \leq  (1+o(1))\frac{n\log{\delta}}{\delta}\;.
$$ This bound is sharp and, in particular, the ``typical'' $\delta$-regular graph is an asymptotically tight example. Indeed, for
such a ``typical'' graph $G$, the upper bound of
Theorem~\ref{thm:watchDelta} gives
\begin{align}\label{eq:2}
w(G)\leq \gamma(G) \lceil\log{\Delta+2}\rceil =
\Omega\left(\frac{n\log^2{\delta}}{\delta}\right)\;.
\end{align}
By Observation~\ref{obs:watch}, a direct corollary of
Theorem~\ref{thm:stronger} is the following:

\begin{corollary}\label{cor:WS}
 For any graph $G$ on $n$ vertices with minimum degree $\delta\geq 66\log\Delta$ and
 maximum degree $\Delta=\omega(1)$, we have:
$$
w(G) \leq 134\frac{n\log{\Delta}}{\delta}\;.
$$
\end{corollary}

Note that this bound improves Theorem~\ref{thm:watchDelta} when the
maximum degree is $\Delta=\Poly(\delta)$.

\section{Concluding remarks and open questions}

\textbf{1.} The kind of results we provide in this paper can be
connected to the notion of resilience. Given a graph property $\P$,
the \emph{global resilience} of $G$ with respect to $\P$ is the
minimum number of edges one has to delete to obtain a graph not
satisfying $\P$. The resilience of monotone properties is well
studied, in particular, in the context of random graphs~\cite{SV08}.

Our result can be interpreted in terms of the resilience of the
following (non-monotone) property $\P$: ``$G$ has a large identifying
code number in terms of its degree parameters, $\delta$ and
$\Delta$''. For any graph $G$ satisfying the hypothesis
$\Delta=\omega(1)$ and $\delta\geq 66\log{\Delta}$,
Theorem~\ref{thm:stronger} can be stated as: the resilience of $G$
with respect to $\P$ is $O(n\log{\Delta})$. Moreover,
Corollary~\ref{cor:disjointcliques} shows that there are graphs that
attain this value of the resilience.
\vspace{0.3cm}

\textbf{2.} In Theorem~\ref{thm:stronger}, we show the existence of a
small identifying code for a large spanning subgraph of $G$. However,
our proof is not constructive and, besides, the probability that such
pair exists is exponentially small, due to the use of the local
lemma. The algorithmic version of the local lemma proposed by Moser
and Tardos, allows to explicitly find a configuration that avoids all
the bad events $E_{uv}$, when these events are determined by a finite
set of mutually independent random variables. Unfortunately, this is
not the case here, since $E_{uv}$ depends on the random variables
determining the existence of certain edges close to $uv$. These random
variables are not independent because of the definition of $p_{uv}$.

On the other hand, if we do not want to argue in terms of the maximum
degree $\Delta$, one can show that by deleting a set of $O(n\log{n})$
random edges we have an identifying code of size
$O\left(\frac{n\log{n}}{\delta}\right)$ with probability $1-o(1)$. In
such a case, the proof provides a randomized algorithm which
constructs the desired code for almost all subgraphs.
\vspace{0.3cm}

\textbf{3.} Note that a notion similar to identifying codes,
\emph{locating-dominating sets}, was also extensively studied in the
literature (see e.g.~\cite{biblio} for many references). A set $C$ of
vertices of $G$ is a locating-dominating set if $C$ is a dominating
set which separates all pairs of vertices in $V(G)\setminus C$. It
follows that any identifying code is a locating-dominating set, hence
Theorem~\ref{thm:stronger} also holds for this notion. In fact, the
proof of Corollary~\ref{cor:disjointcliques} can be adapted for this
case too.

\textbf{4.} As further research, it would be very interesting to close the gap between the result in
Theorem~\ref{thm:stronger} and the lower bound given by the example
in Corollary~\ref{cor:disjointcliques}. Motivated by this example, we ask the following question:
\begin{question}
Is it true that for any graph $G$ with minimum degree $\delta$, there
exists a subset of edges $F\subset E(G)$ of size
$$
|F|= O\left(n\log{\delta}\right)\;,
$$ 
such that 
$$
\M(G\setminus F)=O\left(\frac{n\log{\delta}}{\delta}\right)\;?
$$
\end{question}
It seems to us  that the techniques used in this paper will not  provide an answer to
the previous question. The main obstacle is the use of the local lemma, which forces us to take into
account the role of the maximum degree of $G$.





\end{document}